\documentclass{amsart}

\usepackage[ascii]{inputenc}
\usepackage{amsmath,amsfonts,amssymb,amsthm}
\usepackage{hyperref,graphicx,cleveref,enumitem}

\usepackage[dvipsnames]{xcolor}

\theoremstyle{plain}
\newtheorem{theorem}{Theorem}
\newtheorem{lemma}[theorem]{Lemma}
\newtheorem{proposition}[theorem]{Proposition}
\newtheorem{corollary}{Corollary}

\theoremstyle{remark}
\newtheorem{remark}{Remark}
\newtheorem*{remark*}{Remark}

\theoremstyle{definition}
\newtheorem*{definition*}{Definition}

\newtheorem*{problem*}{Problem}

\def\N{\mathbb{N}}
\def\R{\mathbb{R}}
\def\Z{\mathbb{Z}}

\def\H{{L^2(\T^d)}}

\def\W{\mathcal{W}}

\def\T{\mathbb{T}}

\def\Div{{\rm div}}

\renewcommand{\epsilon}{\varepsilon}

\let\div\relax
\DeclareMathOperator*{\div}{div}

\newcommand{\norm}[1]{{\left\Vert {#1}\right\Vert}}

\DeclareMathOperator*{\sinc}{sinc}

\begin{document}
\title[Calder\'on's inverse problem with finitely many measurements]{Calder\'on's Inverse Problem\\ with a Finite Number of Measurements}

\author{Giovanni S. Alberti}
\address{Department of Mathematics, University of Genoa, Via Dodecaneso 35, 16146 Genova, Italy.}
\email{alberti@dima.unige.it}

\author{Matteo Santacesaria}
\address{Department of Mathematics and Statistics, University of  Helsinki, 00014, Finland.}
\email{matteo.santacesaria@helsinki.fi}


\date{March 12, 2018}

\begin{abstract}
We prove that an $L^\infty$ potential in the Schr\"odinger equation in three and higher dimensions can be uniquely determined from a finite number of boundary measurements, provided it belongs to a known finite dimensional subspace $\W$. As a corollary, we obtain a similar result for Calder\'on's inverse conductivity problem. Lipschitz stability estimates and a globally convergent nonlinear reconstruction algorithm for both inverse problems are also presented. These are the first results on global uniqueness, stability and reconstruction for nonlinear inverse boundary value problems with finitely many measurements. We also discuss a few relevant examples of finite dimensional subspaces $\W$, including bandlimited and piecewise constant potentials, and explicitly compute the number of required measurements as a function of $\dim \W$.
\end{abstract}

\maketitle

\section{Introduction and main results}
Consider the Schr\"odinger equation
\begin{equation}\label{eq:schr}
(-\Delta +q ) u = 0 \qquad \text{in } \Omega,
\end{equation}
where $\Omega \subseteq \R^d$, $d \geq 3$, is an open bounded domain with Lipschitz boundary and $q \in L^{\infty}(\Omega)$ is a potential. Assuming that 
\begin{equation}\label{hyp:dir}
0\; \text{is not a Dirichlet eigenvalue for } -\Delta + q  \text{ in } \Omega,
\end{equation}
 it is possible to define the Dirichlet-to-Neumann (DN) map
\begin{equation*}
\Lambda_q \colon  u|_{\partial \Omega} \mapsto \left.\frac{\partial u}{\partial \nu}\right|_{\partial \Omega},
\end{equation*}
where $\nu$ is the unit outward normal to $\partial \Omega$,  defined as an operator $\Lambda_q \colon H^{1/2}(\partial \Omega) \to H^{-1/2}(\partial \Omega)$.

\subsection{Global uniqueness}
The Gel'fand-Calder\'on's inverse problem asks if it is possible to determine the potential $q$ from the knowledge of its associated DN map $\Lambda_q$. A more realistic version of the problem is the following: assuming that $q$ belongs to a known finite dimensional subspace of $L^\infty(\Omega)$, is it uniquely determined from a finite number of boundary measurements? Our main result gives a positive answer to this question.
\begin{theorem}\label{theo:main}
Take $d\ge 3$ and let $\Omega \subseteq \R^d$ be a  bounded Lipschitz domain and $\W \subseteq L^\infty(\Omega)$ be a finite dimensional subspace. There exists $N\in\N$ such that for any $R>0$ and $q_1 \in \W$ satisfying $\norm{q_1}_{L^\infty(\Omega)}\le R$ and \eqref{hyp:dir}, the following is true.

There exist $\{f_l\}_{l=1}^N \subseteq H^{1/2}(\partial \Omega)$   such that for any $q_2 \in \W$ satisfying $\norm{q_2}_{L^\infty(\Omega)} \le R$ and \eqref{hyp:dir}:
\begin{center}
if $\Lambda_{q_1} f_l = \Lambda_{q_2} f_l$ for $l=1,\ldots,N$, then $q_1 = q_2$.
\end{center}
\end{theorem}

\begin{remark}\label{rem:balancing}
The dependence of $N$ on $\W$ is explicit in the proof, see \eqref{eq:balancing}.
Thus, in principle, it is always possible to determine $N$ given the subspace $\W$ (see \Cref{sec:examples} for three relevant examples).
\end{remark}

\begin{remark}\label{rem:dir}
The assumption that $0$ is not a Dirichlet eigenvalue for $-\Delta + q$ in $\Omega$ has been made only to define $\Lambda_q$, which simplifies the exposition, and it can be removed. See Remark~\ref{rem:dir2} below for more details.
\end{remark}

Theorem~\ref{theo:main} readily yields a similar result for Calder\'on's inverse conductivity problem \cite{calderon1980}, which concerns the determination of an electrical conductivity $\sigma \in L^{\infty}(\Omega)$ satisfying
\begin{equation}\label{eq:elliptic}
\lambda^{-1}\le \sigma\le \lambda\quad\text{almost everywhere in $\Omega$}
\end{equation} 
for some $\lambda>1$, from the DN map
\[
\Lambda_{\sigma} : u|_{\partial \Omega} \mapsto \left.\sigma \frac{\partial u}{\partial \nu}\right|_{\partial \Omega},
\]
where $u$ solves the conductivity equation $-\Div(\sigma \nabla u) = 0$ in $\Omega$. This is the mathematical model for electrical impedance tomography (EIT).
\begin{corollary}\label{cor:uni}
Take $d\ge 3$ and let $\Omega \subseteq \R^d$ be a  bounded Lipschitz domain and $\W \subseteq L^\infty(\Omega)$ be a finite dimensional subspace.
There exists $N\in\N$ such that for any $\lambda>1$ and $\sigma_1 \in W^{2,\infty}(\Omega)$ satisfying  \eqref{eq:elliptic} and such that $\frac{\Delta \sqrt{\sigma_1}}{\sqrt{\sigma_1}} \in \W $ and $\sigma_1 =1$ in a neighborhood of $\partial \Omega$, the following is true.

There exist $\{f_l\}_{l=1}^N \subseteq H^{1/2}(\partial \Omega)$   such that for any $\sigma_2 \in W^{2,\infty}(\Omega)$ satisfying  \eqref{eq:elliptic} and such that $\frac{\Delta \sqrt{\sigma_2}}{\sqrt{\sigma_2}} \in \W $ and $\sigma_2 =1$ in a neighborhood of $\partial \Omega$:
\begin{center}
if $\Lambda_{\sigma_1} f_l = \Lambda_{\sigma_2} f_l$ for $l=1,\ldots,N$, then $\sigma_1 = \sigma_2$.
\end{center}
\end{corollary}

To our knowledge, these are the first uniqueness results for the Gel'fand-Calder\'on and Calder\'on problems with a finite number of measurements. The only previous result of this kind is \cite{friedman1989}, where it was shown that a single boundary measurement was enough to determine a piecewise constant conductivity with discontinuities on a single convex polygon. All other uniqueness results rely on an infinite number of measurements, even for potentials belonging to known finite dimensional subspaces. Some fundamental contributions to the two problems include \cite{Sylvester1987,Novikov1988,Nachman1996,astala2006,bukhgeim2008,
haberman2015,caro2016,lakshtanov2016} for global uniqueness and reconstruction, and
\cite{alessandrini1988,clop2010,novikov2010,blaasten2015} for global stability. An interesting uniqueness result from a finite number of measurements is \cite{blaasten2017}, for a related inverse problem. 

\subsection{Lipschitz stability and reconstruction}
We are also able to prove Lipschitz stability estimates for the two problems.
\begin{theorem}\label{theo:stab}
Take $d\ge 3$ and let $\Omega \subseteq \R^d$ be a  bounded Lipschitz domain and $\W \subseteq L^\infty(\Omega)$ be a finite dimensional subspace. There exists $N\in\N$ such that for every $R,\alpha>0$ and $q_1 \in \W$ satisfying $\norm{q_1}_{L^\infty(\Omega)}\le R$ and \eqref{hyp:dir}, the following is true.

There exist $\{f_l\}_{l=1}^N \subseteq H^{1/2}(\partial \Omega)$ such that for every $q_2 \in \W$ satisfying $\norm{q_2}_{L^\infty(\Omega)} \le R$ and \eqref{hyp:dir}, we have
\begin{equation*}
\|q_2 - q_1\|_{L^2(\Omega)} \leq e^{C N^{\frac{1}{2}+\alpha}}
\norm{\left(\Lambda_{q_2}f_l - \Lambda_{q_1}f_l\right)_{l=1}^N}_{H^{-1/2}(\partial \Omega)^N}
\end{equation*}
for some $C>0$ depending only on $\Omega$, $R$ and $\alpha$.
\end{theorem}

\begin{corollary}\label{cor:stab}
Take $d\ge 3$ and let $\Omega \subseteq \R^d$ be a  bounded Lipschitz domain and $\W \subseteq L^\infty(\Omega)$ be a finite dimensional subspace.
There exists $N\in\N$ such that for any $\lambda>1$, $\alpha>0$ and $\sigma_1 \in W^{2,\infty}(\Omega)$ satisfying  \eqref{eq:elliptic} and such that $\frac{\Delta \sqrt{\sigma_1}}{\sqrt{\sigma_1}} \in \W $ and $\sigma_1 =1$ in a neighborhood of $\partial \Omega$, the following is true.

There exist $\{f_l\}_{l=1}^N \subseteq H^{1/2}(\partial \Omega)$   such that for any $\sigma_2 \in W^{2,\infty}(\Omega)$ satisfying  \eqref{eq:elliptic} and such that $\frac{\Delta \sqrt{\sigma_2}}{\sqrt{\sigma_2}} \in \W $ and $\sigma_2 =1$ in a neighborhood of $\partial \Omega$, we have
\begin{equation*}
\|\sigma_2 - \sigma_1\|_{L^2(\Omega)} \leq e^{C N^{\frac{1}{2}+\alpha}}\norm{\left(\Lambda_{\sigma_2}f_l - \Lambda_{\sigma_1}f_l\right)_{l=1}^N}_{H^{-1/2}(\partial \Omega)^N}
\end{equation*}
for some $C>0$ depending only on $\Omega$, $\lambda$ and $\alpha$.
\end{corollary}

These are the first stability estimates for the Gel'fand-Calder\'on and Calder\'on problems with a finite number of measurements. Lipschitz stability results have been previously known only when an infinite number of measurements are available \cite{2005-alessandrini-vessella,2011-beretta-francini,beretta2013,gaburro2015,beretta2017,alessandrini2017,
alessandrini2017u} The exponentially growing constant is coherent with the exponential instability of the problem \cite{mandache2001,dicristo2003,isaev2011,2006-rondi}.

We finally employ the same ideas to present a new nonlinear iterative reconstruction algorithm for the two inverse problems and prove that it is globally convergent in \Cref{theo:rec}. Given $(f_l,\Lambda_q f_l)_{l=1}^N$, the algorithm constructs a sequence $q_n \to q$ in $L^2(\Omega)$ where $q_0 \in \W$ is \textit{any} initial guess. The algorithm converges exponentially and its stability is given by \Cref{theo:stab} and \Cref{cor:stab}. The details are presented in Section~\ref{sec:rec}. 

Note that this represents the first globally convergent iterative algorithm for the Gel'fand-Calder\'on and Calder\'on problems from a finite number of measurements. All reconstruction algorithms used so far have been either based on the full DN map, either locally convergent or with no proof of convergence. We refer to \cite{mueller2012} for an extensive review on reconstruction methods for nonlinear inverse problems.\smallskip

The strategy of the proof of Theorems~\ref{theo:main} and \ref{theo:stab} is as follows. We use Alessandrini's identity to trasform the boundary data into integral measurements, and use complex geometrical optics (CGO) solutions to construct a nonlinear operator $U$, expressing the DN map in a more convenient form. We then write $U$ as $U=F+B$ where $F$ is the Fourier transform and $B$  is a nonlinear term. We show that $B$ is a contraction, provided the CGO solutions are constructed for sufficiently high complex frequencies. Finally, the problem is reduced to a fixed-point problem involving a nonlinear Fourier transform.

\subsection{Open questions and comments}
This paper has been greatly motivated by potential applications of ideas coming from applied harmonic analysis and sampling theory to  inverse problems in PDE. This approach paves the way for several interesting research directions and open problems. Let us mention some of them.
\begin{itemize}[leftmargin=10pt]
\item The inspiration for this paper came from a previous paper of the authors \cite{alberti2017infinite}, in which the theory of compressed sensing (CS) was generalized to the infinite dimensional setting for linear operators which are not necessarily isometries, in order to make it more flexible for the applications to inverse problems in PDE. The current paper shows that the effects of the non-linearity of the inverse problem may be mitigated by carefully selecting the measurements. Thus, we expect that the theory of CS may be applied to this nonlinear problem as well: this would give that the number of measurements needed are (substantially) proportional to the sparsity of the unknown.
\item In this paper, the boundary Dirichlet data $\{f_l\}_l$ are chosen dependently of the unknown potential $q$. It is natural to wonder whether, using the assumption that $q$ belongs to a finite dimensional space, it is possible to determine (a possibly larger number of) $\{f_l\}_l$ independently of $q$.
\item In Corollaries~\ref{cor:uni} and \ref{cor:stab} we  assumed that the conductivities are equal to $1$ in a neighborhood of the boundary. In order to overcome this limitation, one would need to recover the boundary values of a conductivity and of its normal derivative from the boundary data. This is well understood when the full DN map is available \cite{nachman1988,Nachman1996}, but it is still open in case of a finite number of measurements.
\item The results of this paper cannot be directly extended to the two dimensional case. In that setting, the reconstruction methods for the Gel'fand-Calder\'on and the Calder\'on problem are different from the one presented here, and it would be very interesting to study the problem with a finite number of measurements.
\item In \Cref{sec:examples} we compute the number of required measurements $N$ for some choices of subspaces $\W$: it would be interesting to study the optimality of these bounds and to derive similar estimates in other relevant cases.
\item In this paper we consider the continuum model for these inverse boundary value problems. Possible extensions to more realistic and physical models (such as the complete electrode model for EIT) and the numerical analysis and implementation of the reconstruction algorithm presented may be investigated.
\item It is natural to wonder whether $\W$ could be required to be only a finite dimensional submanifold of $L^\infty(\Omega)$ and not necessarily a linear subspace. As we point out in \Cref{rem:manifold}, the current proof does not work in such generality, and new ideas are required.
\item Finally, it is expected that the approach presented in this paper can be extended to other infinite dimensional inverse problems in PDE with finitely many measurements.
\end{itemize}

\subsection{Structure of the paper}
\Cref{sec:proofs} contains the proofs of the results stated above. \Cref{sec:rec} presents the reconstruction algorithm and its convergence properties. In \Cref{sec:examples} we discuss some examples of subspaces $\W$, for which we compute explicitly the number of required measurements $N$ as a function of $\dim\W$.

\section{Proofs}\label{sec:proofs}
Take $q\in L^\infty(\Omega)$ such that $\norm{q}_{L^\infty(\Omega)}\le R$. Given two boundary voltages $f,g \in H^{1/2}(\partial \Omega)$ we have Alessandrini's identity \cite{alessandrini1988}:
\begin{equation}\label{eq:alessandrini}
\langle g, (\Lambda_q - \Lambda_{0})f\rangle_{{ H^{\frac12}(\partial \Omega) \times  H^{-\frac12}(\partial \Omega)}} = \int_{\Omega} q\, u^0_g u_f\,dx,
\end{equation}
where $u_f$ (resp.\ $u^0_g$) solves the Schr\"odinger equation \eqref{eq:schr} with potential $q$ (resp.\ $0$) and Dirichlet data $f$ (resp.\ $g$). The quantity on the left of this identity is known since $\Lambda_q f$ is the boundary measurement corresponding to the chosen potential $f$ and $\Lambda_0$ is the DN map corresponding to the unperturbed Laplacian. This identity allows to transform the boundary data into measurements of scalar products between the unknown $q$ and some test functions to be suitably chosen. 

\begin{remark}\label{rem:dir2}
As anticipated in \Cref{rem:dir}, assumption \eqref{hyp:dir} was made only to define $\Lambda_q$, but can be removed. Indeed, it is possible to consider a finite number of Cauchy data as boundary measurements, i.e.\ $( u_l|_{\partial \Omega},\frac{\partial u_l}{\partial \nu}|_{\partial \Omega} )_{l=1}^N$, where $(-\Delta + q)u_l = 0$ in $\Omega$, $l=1,\ldots, N$. In this case the same uniqueness and stability results hold, provided one makes use of the results of \cite{isaev2013}, where Alessandrini's type identities were obtained for a general impedance boundary map, also known as Robin-to-Robin map. This is a generalization of the DN map, and it can be defined even when $0$ is a Dirichlet eigenvalue for the Sch\"odinger operator.
\end{remark}

Without loss of generality, we can assume that $\Omega\subseteq\T^d$, where $\T=[0,1]$, and in the following we will extend any function of $L^\infty(\Omega)$ to $L^\infty(\T^d)$ by zero. Fix an arbitrary parameter $p\in (d,+\infty)$. In the rest of this subsection, with an abuse of notation, several different positive constants depending only on $d$, $p$ and $R$ will be denoted by the same  letter $c$. 

We now construct a special class of solutions to \eqref{eq:schr} as in  the seminal paper \cite{Sylvester1987}. For  $k \in \Z^d$, choose $\eta, \xi \in \R^d$ such that $|\xi| = |\eta|= 1$ and $\xi \cdot \eta = \xi \cdot k = \eta \cdot k = 0$. For $t\in\R$ define
\begin{align}
\label{def:zeta1}
\zeta_1^{k,t} = -i (\pi  k + t\xi) + \sqrt{t^2 +\pi^2|k|^2}\eta,\\ 
\label{def:zeta2} 
\zeta_2^{k,t} = -i(\pi k - t\xi) - \sqrt{t^2 +\pi^2|k|^2}\eta,
\end{align}
so that
\[
\zeta_j^{k,t}\cdot \zeta_j^{k,t} = 0 \quad \text{for } j =1,2,  \qquad \zeta_1^{k,t}+ \zeta_2^{k,t} =- 2\pi i k.
\]
For every $t\ge c$ we can construct a solution $\psi^{k,t}$ of \eqref{eq:schr} in $\R^d$ (with $q$ extended to $\R^d$ by zero) of the form
\[
\psi^{k,t}(x) = \psi(x,\zeta_1^{k,t}) = e^{\zeta_1^{k,t}\cdot x}(1+r^{k,t}(x)),\qquad x\in \R^d,
\]
in which the error term $r^{k,t}$ satisfies the estimates
\begin{align}
\label{eq:est_r}
&\|r^{k,t}\|_{L^2(\T^d)} \leq \frac{c}{t}, \\
\label{eq:est_r1}
&\|r^{k,t}\|_{L^p(\T^d)} \leq c,\\
 \label{eq:est_rgrad} &\|\nabla r^{k,t}\|_{L^2 (\T^d)}\leq c.
\end{align}
The first two bounds are given in \cite{paivarinta2004} (see Theorem~4.1 and Lemma~5.5). The third estimate is not explicitly mentioned in this paper, but follows immediately from \eqref{def:zeta1} and \eqref{eq:est_r} and using standard energy estimates for elliptic equations\footnote{
The error term $r^{k,t}$ solves $-\Delta r^{k,t} = \div(\zeta_1^{k,t} r^{k,t}) +q(r^{k,t}-1)$ in a ball $B\supseteq\T^d$, and so
\[
\|\nabla r^{k,t}\|_{L^2(\T^d)} \le c\left(|\zeta_1^{k,t}| \|r^{k,t}\|_{L^2(B)} +R \|r^{k,t}-1\|_{L^2(B)}+\|r^{k,t}\|_{L^2(B)}\right) \le c,
\]
where we used estimate \eqref{eq:est_r} in $B$.
}.

The solutions $\psi^{k,t}$ are known as exponentially growing solutions, Faddeev-type solutions \cite{faddeev1965} or complex geometrical optics  solutions. Note that $e^{\zeta_2^{k,t_{k}}\cdot x}$ is harmonic in $\R^d$.

It is useful to consider an ordering of the frequencies in $\Z^d$, namely a bijective map $\rho\colon\N\to \Z^d$, $l\mapsto k_l$.
For each $k \in \Z^d$ fix {$t_k\ge c$} and define the measurement operator $U\colon L^\infty(\T^d) \to \ell^\infty$ by
\begin{equation}\label{def:U_EIT}
\begin{split}
(U ( q))_l&:= \int_{\T^d} q(x) e^{\zeta_2^{k_l,t_{k_l}}\cdot x} \, \psi^{k_l,t_{k_l}}(x) \,dx \\
&= \langle  e^{\zeta_2^{k_l,t_{k_l}}\cdot x}, (\Lambda_q - \Lambda_{0})\psi^{k_l,t_{k_l}} \rangle_{{ H^{\frac12}(\partial \Omega) \times  H^{-\frac12}(\partial \Omega)}},
\end{split}
\end{equation}
where the second identity follows from \eqref{eq:alessandrini} (and is valid only whenever \eqref{hyp:dir} holds true and $q=0$ almost everywhere in $\T^d\setminus\Omega$). Note that the operator $U$ is nonlinear, since the solution $\psi^{k_l,t_{k_l}}$ depends on $q$. In the literature, this operator is also known as nonlinear Fourier transform or generalized scattering transform or amplitude. Using the same ordering of $\Z^d$, we define the discrete Fourier transform $F\colon \H \to \ell^2$ by
\begin{equation}\label{def:DFT}
F( q):=\int_{\T^d} q(x)e^{-2\pi ik_l\cdot x}\,dx=(\langle q, e_{k_l} \rangle)_l,
\end{equation}
where $e_k(x)=e^{2\pi ik\cdot x}$, and the nonlinear operator $B: L^\infty(\T^d) \to \ell^\infty$  by
\begin{equation*}
B(q) := (\langle q, e_{k_l}r_l \rangle)_l,
\end{equation*}
where $r_l = r^{k_l,t_{k_l}}$, so that $U = F +B$.

The first global uniqueness proof for the 3D Calder\'on problem \cite{Sylvester1987} consisted in showing that $U \to F$ as the parameter $t \to +\infty$. In our case we cannot do that, since we want to be able to select a finite number of measurements from the operator $U$, which depends on the choice of the fixed parameters $t_{k_l}$.

In order to prove a uniqueness result with a finite number of measurements we will use a fixed-point argument, based on the following lemma.

\begin{lemma}\label{lem:contr}
Take $s>\frac{dp}{2(p-d)}$. There exists $c'>0$ depending only on $d$, $s$, $p$ and $R$, such that if $t_k= c'(|k|^s+1)$ for every $k\in\Z^d$ then
\[
\|B(q_2) - B(q_1)\|_{\ell^2} \leq \frac12 \|q_2 - q_1\|_\H
\]
for every $q_1,q_2\in L^\infty(\T^d)$ such that $\|q_i\|_{L^\infty(\T^d)} \leq R$, $i = 1,2$. In other words, the operator $B$ restricted to the closed ball $\overline{B}(0,R)$ of $L^{\infty}(\T^d)$ is a contraction.
\end{lemma}
\begin{proof}
We have, from the definition of $B$,
\begin{align*}
(B(q_2) - B(q_1))_l = \langle q_2 - q_1, e_l r_l(q_2) \rangle + \langle q_1, e_l ( r_l(q_2)-r_l(q_1))\rangle,
\end{align*}
where we called $e_l = e_{k_l}$ and emphasized the dependence on $q_i$. By \eqref{eq:est_r} and the assumption $t_k = c'(|k|^s+1)$, we obtain
\begin{equation}\label{eq:first}
\begin{split}
|(B(q_2) - B(q_1))_l| &\leq \!\|q_2 - q_1\|_\H \| r_l(q_2)\|_\H \!+ \!\| q_1\|_\H \| r_l(q_2)-r_l(q_1)\|_\H\\
&\leq \|q_2 - q_1\|_\H \frac{c}{c' (|k_l|^s+1)} +R\| r_l(q_2)-r_l(q_1)\|_\H.
\end{split}
\end{equation}

Now, following \cite{Sylvester1987}, let $L w = \Delta w + \zeta_1^{k_l,t_{k_l}} \cdot \nabla w$. The remainders $r_l(q_i)$ satisfy the equations
\[
L r_l(q_i) -q_i r_l(q_i) = q_i, \qquad i = 1,2,
\]
so that the difference $r_l = r_l(q_2)-r_l(q_1)$ satisfies
\begin{equation}\label{eq:difr}
L r_l -q_2 r_l = (q_2-q_1)(1+r_l(q_1))\quad \text{in $\R^d$}.
\end{equation}
Applying \cite[Theorem~4.1]{paivarinta2004} to \eqref{eq:difr} gives
\begin{equation*}
\begin{split}
\| r_l(q_2)-r_l(q_1)\|_\H &\leq c \frac{\|(q_2 -q_1)(1+r_l(q_1))\|_{L^{\frac{2p}{p+2}}(\T^d)}}{|t_{k_l}|^{1-\frac d p}}\\
&
\leq c \frac{\|(q_2 -q_1)\|_\H\|1+r_l(q_1)\|_{L^{p}(\T^d)}}{|t_{k_l}|^{1-\frac d p}}.
\end{split}
\end{equation*}
As a consequence, by \eqref{eq:est_r1} we obtain
\[
\| r_l(q_2)-r_l(q_1)\|_\H \leq \frac{c}{\bigl(c' (|k_l|^s+1)\bigr)^{1-\frac{d}{p}}} \|q_2-q_1\|_\H.
\]
Inserting this estimate into \eqref{eq:first} yields
\[
|(B(q_2) - B(q_1))_l| \leq \frac{c}{\bigl(c' (|k_l|^s+1)\bigr)^{1-\frac{d}{p}}} \|q_2-q_1\|_\H .
\]
Since $2s-\frac{2sd}{p}>d$, the series
\[
\gamma_s^2:=\sum_{l\in\N} \frac{1}{(|k_l|^s+1)^{2-\frac{2d}{p}}} = 
\sum_{k\in\Z^d} \frac{1}{(|k|^s+1)^{2-\frac{2d}{p}}}
\]
is convergent, and so we finally obtain
\[
\|B(q_2) - B(q_1)\|_{\ell^2} \leq \frac{c\gamma_s }{(c')^{1-\frac{d}{p}}} \|q_2-q_1\|_\H .
\]
Choosing $c' \ge  (2c\gamma_s)^{\frac{p}{p-d}}$ yields the desired result.
\end{proof}

The next result shows that, given a known finite dimensional subspace $\W$ of $L^\infty(\Omega)$, there exists a number $N$ such that two potentials in $\W$ satisfying $(U(q_1))_l = (U(q_2))_l$ for $l = 1, \dots, N$ must coincide. Let $P_N\colon\ell^\infty\to\ell^\infty$ be the projection onto the first $N$ components, namely $P_N(a_1,a_2,\dots)=(a_1,\dots,a_N,0,0,\dots)$, and $P_\W\colon L^2(\T^d)\to L^2(\T^d)$ be the orthogonal projection onto $i(\W)$, where $i\colon L^\infty(\Omega)\to L^2(\T^d)$ is the extension operator by zero.

\begin{proposition}\label{prop:main}
Fix $t_k= c'(|k|^s+1)$ for every $k\in\Z^d$ as in Lemma~\ref{lem:contr}. There exists $N\in\N$ depending only on $\W$ such that the following is true.
For any $q_1, q_2 \in \W$ with $\norm{q_i}_{L^\infty(\T^d)}\le R$ we have
\begin{equation*}
\norm{q_1-q_2}_\H \le 4\norm{P_NU(q_1)-P_NU(q_2)}_{\ell^2}.
\end{equation*}
\end{proposition}
\begin{proof}
From the identity $U=F+B$ and the assumptions on $q_1,q_2$, we readily obtain
\[
\begin{split}
P_NU(q_1)-P_NU(q_2)&=P_NF(q_1-q_2)+P_N(B(q_1)-B(q_2))\\
&=F(q_1-q_2)-P_N^\perp F(q_1-q_2)+P_N(B(q_1)-B(q_2)),
\end{split}
\]
where $P_N^\perp=I-P_N$, and thus
\[
F(q_1-q_2)=(P_NU(q_1)-P_NU(q_2)) + P_N^\perp F(q_1-q_2) - P_N(B(q_1)-B(q_2)).
\]
Using the fact that $F$ is unitary  and Lemma~\ref{lem:contr} we have
\[
\begin{split}
\|q_1-q_2&\|_\H = \|F(q_1-q_2)\|_{\ell^2} \\
&\leq \|P_N U(q_1)-P_N U(q_2)\|_{\ell^2}+\|P_N^\perp F(q_1-q_2)\|_{\ell^2} +\|P_N (B(q_2)-B(q_1))\|_{\ell^2}\\
&\leq \|P_N U(q_1)-P_N U(q_2)\|_{\ell^2}+\|P_N^\perp F(q_1-q_2)\|_{\ell^2}+\frac 1 2 \|q_1-q_2\|_\H,
\end{split}
\]
which gives, using the fact that $q_1,q_2 \in \W$,
\begin{equation*}
\|q_1-q_2 \|_\H \leq 2\|P_N^\perp F P_\W(q_1-q_2)\|_{\ell^2} + 2 \|P_N U(q_1)-P_N U(q_2)\|_{\ell^2}.
\end{equation*}

Now, since $P_N^\perp \to 0$ strongly as $N \to \infty$ and $FP_\W$ is a finite rank operator, there exists $N$ such that 
\begin{equation}\label{eq:balancing}
\|P_N^\perp F P_\W\|_{\H \to \ell^2} \leq \frac 1 4.
\end{equation}
This immediately yields the final estimate
\begin{equation*}
\|q_1-q_2\|_\H \leq 4 \|P_N U(q_1)-P_N U(q_2)\|_{\ell^2}. \qedhere
\end{equation*}
\end{proof}

\begin{remark}\label{rem:manifold}
It is natural to wonder whether one may extend the uniqueness result presented in this paper to the case when $\W$ is a finite dimensional submanifold of $L^\infty(\Omega)$. While this remains a very interesting open question to investigate, it is clear that, in such generality, the current proof would not work. Indeed, when $\Omega=\T^d$, for the one-dimensional manifold
\[
\W=\{x\mapsto e^{2\pi i \xi x_1}:\xi\in\R\}\subseteq L^\infty(\T^d),
\]
we immediately have that
\[
\sup_{q\in \W} \norm{P_N^\perp F q}_{\ell^2} = 1,
\]
and so an inequality like \eqref{eq:balancing} cannot hold.
\end{remark}

The next result is a more precise version of Theorem~\ref{theo:main}, where we show how having the same boundary measurements yields $P_N U(q_1) = P_N U(q_2)$.

\begin{theorem}\label{theo:main2}
Take $d\ge 3$ and let $\Omega \subseteq \R^d$ be a  bounded Lipschitz domain and $\W \subseteq L^\infty(\Omega)$ be a finite dimensional subspace. There exists $N\in\N$ such that the following is true.

Take $R>0$ and $q_1,q_2 \in \W$ satisfying $\norm{q_j}_{L^\infty(\Omega)}\le R$ and \eqref{hyp:dir}  ($j=1,2$) and let $f_l = \psi_1^{k_l,t_{k_l}}|_{\partial \Omega}$, where $\psi_1^{k_l,t_{k_l}}$, $l \in \N$, are the CGO solutions corresponding to $q_1$.
\begin{center}
If $\Lambda_{q_1} f_l = \Lambda_{q_2} f_l$ for $l=1,\ldots,N$, then $q_1 = q_2$.
\end{center}
\end{theorem}
\begin{proof}
Let $\psi_j^l = \psi_j^{k_l,t_{k_l}}$ be the CGO solutions corresponding to $q_j$, $j=1,2$ and $N$ be as in Proposition~\ref{prop:main}. We claim that $\psi_1^l|_{\partial \Omega} = \psi_2^l|_{\partial \Omega}$, for $l =1,\dots, N$. Indeed, $\psi_j^l|_{\partial \Omega}$ can be characterized as the unique solution in $H^{1/2}(\partial \Omega)$ of the boundary integral equation
\begin{equation}\label{eq:BIE}
\psi_j^l(x) = e^{\zeta_1^{k_l,t_{k_l}}\cdot x} + \int_{\partial \Omega} G(x-y,\zeta_1^{k_l,t_{k_l}})(\Lambda_{q_j} - \Lambda_0)\psi_j^l(y) d\sigma (y),\quad x \in \partial \Omega,
\end{equation}
for $j=1,2$, where $G(x,\zeta)$ is the Faddeev-Green function: see \cite[Proposition 2]{Novikov1988}, \cite[Theorem 1.4]{nachman1988}, and \cite[Theorem 5]{Nachman1996}. Note that we only need \cite{Nachman1996} to extend the results of \cite{nachman1988} to Lipschitz domains and to take $H^{1/2}(\partial \Omega)$ as domain for the boundary integral equation \eqref{eq:BIE}.

Thus, for $l = 1,\dots, N$, since $\Lambda_{q_1} (\psi_1^l) = \Lambda_{q_2} (\psi_1^l)$, $\psi_1^l|_{\partial \Omega}$ satisfies
\begin{equation*}
\psi_1^l(x) = e^{\zeta_1^{k_l,t_{k_l}}\cdot x} + \int_{\partial \Omega} G(x-y,\zeta_1^{k_l,t_{k_l}})(\Lambda_{q_2} - \Lambda_0)\psi_1^l(y) d\sigma (y),\quad x \in \partial \Omega,
\end{equation*}
which yields $\psi_1^l|_{\partial \Omega} = \psi_2^l|_{\partial \Omega}$ because of the unique solvability of \eqref{eq:BIE} for $j =2$. This readily gives
\begin{align*}
(U(q_1))_l &=\langle  e^{\zeta_2^{k_l,t_{k_l}}\cdot x}, (\Lambda_{q_1} - \Lambda_{0})\psi_1^l \rangle_{{ H^{\frac12}(\partial \Omega) \times  H^{-\frac12}(\partial \Omega)}} \\
& = \langle  e^{\zeta_2^{k_l,t_{k_l}}\cdot x}, (\Lambda_{q_2} - \Lambda_{0})\psi_2^l \rangle_{{ H^{\frac12}(\partial \Omega) \times  H^{-\frac12}(\partial \Omega)}} \\
& = (U(q_2))_l, 
\end{align*}
for $l = 1,\dots,N$, i.e.\ $P_N U(q_1) = P_N U(q_2)$. Finally, by Proposition~\ref{prop:main} we obtain $q_1=q_2$.
\end{proof}

We now pass to the proof of the stability estimate.

\begin{proof}[Proof of Theorem \ref{theo:stab}] During the proof, several positive constants depending only on $\Omega$, $R$, $\alpha$ and $\rho$ will be denoted by the same letter $C$. As in Theorem~\ref{theo:main2}, we let $f_l = \psi_1^{k_l,t_{k_l}}|_{\partial \Omega}$, for $l =1,\ldots,N$, where we make the choices $s=\frac{d}{2}+\alpha d$ and $p=d+1+\frac{d}{2\alpha}$, so that $s>\frac{dp}{2(p-d)}$. We also choose a particular ordering $\rho:l\in \N\mapsto k_l\in\Z^d$ of the frequencies: suppose that 
\begin{equation}\label{cond:rho}
|k_l| \leq  C_\rho \,l^{1/d} \text{ for some } C_\rho > 0.
\end{equation}

From the definition of $U$,  using \cite[Theorem 1]{novikov2004}, for $l=1,\dots, N$ we have the identity:
\begin{equation*}
(U(q_1))_l - (U(q_2))_l =\langle \psi_2(\cdot,\zeta_2^{k_l,t_{k_l}}), (\Lambda_{q_1} - \Lambda_{q_2})\psi_1^l\rangle_{{ H^{\frac12}(\partial \Omega) \times  H^{-\frac12}(\partial \Omega)}}.
\end{equation*}
This is a particular case of identity (2.8) of \cite{novikov2004}, with a different notation: $U(q)_l$ corresponds to $h(-i\zeta_1^{k_l,t_{k_l}}, i\zeta_2^{k_l,t_{k_l}})$, $\psi(x,\zeta)$ corresponds to $\psi(x,-i\zeta)$ and the operator $\Lambda_q$ to $\Phi$.

We then readily obtain
\begin{equation*} 
|(U(q_1))_l - (U(q_2))_l| \leq \|\psi_2(\cdot,\zeta_2^{k_l,t_{k_l}}) \|_{H^{1/2}(\partial \Omega)} \| (\Lambda_{q_1} - \Lambda_{q_2})\psi_1^l\|_{H^{-1/2}(\partial \Omega)}.
\end{equation*}
The first term can be bounded using the trace theorem \cite{ding1996}:
\begin{align*}
\|\psi_2(\cdot,\zeta_2^{k_l,t_{k_l}}) \|_{H^{1/2}(\partial \Omega)} &\leq  \|\psi_2(\cdot,\zeta_2^{k_l,t_{k_l}}) \|_{H^{1}( \Omega)} \\
&\leq \| e^{\zeta_2^{k_l,t_{k_l}}\cdot x} (1+r_l(q_2,\zeta_2^{k_l,t_{k_l}}))\|_{H^{1}(\Omega)}\\
&\leq C e^{C |k_l|^s} \left(\| 1 \|_{H^{1}(\Omega)} +\|r_l(q_2,\zeta_2^{k_l,t_{k_l}})\|_{H^{1}(\Omega)}\right)\\
& \le  Ce^{C l^{\frac{s}{d}}},
\end{align*}
where we used 
 \eqref{def:zeta2}, \eqref{cond:rho}, \eqref{eq:est_r}, the assumptions on $t_{k_l}$ and the boundedness of $\Omega$.

We have found
\begin{equation*}
|(U(q_1))_l - (U(q_2))_l| \leq C\, e^{Cl^{\frac{s}{d}}}\| (\Lambda_{q_1} - \Lambda_{q_2})\psi_1^l\|_{H^{-1/2}(\partial \Omega)},
\end{equation*}
for $l=1,\dots, N$, which gives
\begin{align*}
\|P_N U(q_1)-P_N U(q_2)\|_{\ell^2} &\leq C\, e^{CN^{\frac{s}{d}}}\sqrt{\sum_{l=1}^N  \|(\Lambda_{q_1} - \Lambda_{q_2})\psi_1^l\|_{H^{-1/2}(\partial \Omega)}^2}\\
&= C\, e^{CN^{\frac{1}{2}+\alpha}}\norm{\left((\Lambda_{q_1} - \Lambda_{q_2})\psi_1^l\right)_{l=1}^N}_{H^{-1/2}(\partial \Omega)^N},
\end{align*}
where we set 
$
\norm{(\varphi_l)_l}^2_{H^{-1/2}(\partial \Omega)^N}:=\sum_{l=1}^N  \|\varphi_l\|_{H^{-1/2}(\partial \Omega)}^2
$. The proof follows from the last estimate and Proposition~\ref{prop:main}.
\end{proof}

\begin{remark}
The proof of \Cref{theo:stab} makes use of a particular ordering of the frequencies $\rho$ satisfying \eqref{cond:rho}. This yields the explicit stability constant $e^{CN^{\frac{1}{2}+\alpha}}$. However, the stability result remains valid for an arbitrary ordering $\rho$, and the constant becomes
\[
\exp\left(C\max(|k_1|^s,\dots,|k_N|^s)\right),
\]
where $s$ is an arbitrary parameter larger than $\frac d 2$.
\end{remark}

\begin{proof}[Proof of Corollary \ref{cor:uni}]
Using the Liouville transformation $u = \tilde u / \sqrt{\sigma}$, if $u$ solves the conductivity equation $\Div (\sigma \nabla u) =0$, then $\tilde u$ solves the Schr\"odinger equation $(-\Delta+ q) \tilde u=0$, with $q = \Delta \sqrt{\sigma}/\sqrt{\sigma}$.
 
Now, since we have $\sigma_j = 1$ in a neighborhood of $\partial \Omega$, we have that $\Lambda_{\sigma_j} = \Lambda_{q_j}$, $q_j = \frac{\Delta \sqrt{\sigma_j}}{\sqrt{\sigma_j}}$, for $j=1,2$, because of the well-known identity $\Lambda_q = \sigma^{-1/2}(\Lambda_{\sigma} + \frac{1}{2}\frac{\partial \sigma}{\partial \nu}) \sigma^{-1/2}$, which follows from the Liouville transformation.
Theorem~\ref{theo:main} immediately yields $q_1 = q_2$, and since $\sigma_j$ are solutions to $(\Delta - q_j)\sqrt{\sigma_j} = 0$ in $\Omega$ with $\sigma_j|_{\partial \Omega} = 1$, we have $\sigma_1 = \sigma_2$.
\end{proof}

\begin{proof}[Proof of Corollary \ref{cor:stab}]
Arguing as in the proof of \cite[Theorem 1]{alessandrini1988} we have the following identity
\begin{equation*}
\Div ( (\sigma_1 \sigma_2)^{1/2} \nabla \log(\sigma_1 / \sigma_2) ) = 2 (\sigma_1 \sigma_2)^{1/2}(q_1-q_2) \quad \text{in } \Omega.
\end{equation*}
Now, using \eqref{eq:elliptic} we find
\begin{equation*}\label{eq:estsq}
\|\sigma_1-\sigma_2\|_{L^2(\Omega)} \leq C(\lambda,\Omega) \|\log(\sigma_1 / \sigma_2)\|_{L^2(\Omega)} \leq C(\lambda,\Omega) \|q_1-q_2\|_{L^2(\Omega)}.
\end{equation*}
Thus, the proof follows from the estimate of Theorem~\ref{theo:stab} and the fact that $\Lambda_{\sigma_j} = \Lambda_{q_j}$, $j =1,2$.
\end{proof}

\section{Reconstruction}\label{sec:rec}
The results of Section~\ref{sec:proofs} can be used to design an iterative nonlinear reconstruction algorithm and to show that it is globally convergent. For simplicity, we consider directly  $\Omega=\T^d$, but the general case may be  handled by using the extension operator as above.

Given a finite dimensional subspace $\W \subseteq L^{\infty}(\T^d)$ and $R>0$, define
\[
\W_R:=\{q\in \W:\norm{q}_{L^\infty(\T^d)}\le R\},
\]
equipped with the $L^2$ norm. For $N\in \N$ and $y \in \ell^2$, we define the nonlinear operator $A:\W_R \to \W_R$ by
\begin{equation}\label{def:A}
A(q) = P_{\W_R}( F^{-1} y + F^{-1}P_N^\perp Fq - F^{-1}P_N B  (q)),
\end{equation}
where $F$, $B$,  $P_N$ and $P_N^\perp$ were defined in Section~\ref{sec:proofs} and $P_{\W_R}$ is the projection from $L^2(\T^d)$ onto the closed and convex set $\W_R$.

Now let $q \in \W_R $ satisfying \eqref{hyp:dir} be the unknown potential and choose the number of measurements $N \in \N$ as in \eqref{eq:balancing}, i.e.\ so that 
\begin{equation}\label{def:N}
\|P_N^\perp F P_\W\|_{L^2(\T^d) \to \ell^2} \leq \frac 1 4.
\end{equation}
Let $(f_l, \Lambda_q(f_l))_{l=1}^N$ be the given boundary data, where $f_l = \psi^{k_l,t_{k_l}}$ are the CGO constructed in Section~\ref{sec:proofs}, and set $y = P_N U(q)$, which is directly computed from $(f_l, \Lambda_q(f_l))_{l=1}^N$ thanks to \eqref{def:U_EIT}. 

The following nonlinear iterative reconstruction algorithm allows for the recovery of $q$ starting from the data $y$.

\begin{theorem}\label{theo:rec}
Under the above assumptions, let $q_0 \in \W_R$ be any initial guess potential and define the sequence 
\[
q_n = A(q_{n-1}), \qquad  n \geq 1.
\]
Then we have the following convergence result:
\[
\|q- q_n\|_{L^2(\T^d)} \leq 4 \left(\frac 3 4\right)^n\|q_1 - q_0\|_{L^2(\T^d)},\qquad n\ge 1.
\]
\end{theorem}

\begin{proof}
We claim that $A(q)=q$ and that $A$ is a contraction. Indeed, by using the identities  $y = P_N U(q)$ and $U = F + B$ we readily derive  that  $A(q)=q$.
Further, it is a straightforward consequence of the fact that $P_{\W_R}$ is Lipschitz (with constant $1$), of the Hilbert projection theorem, of Lemma~\ref{lem:contr} and of assumption \eqref{def:N} that the operator $A$ is a contraction on $\W_R$, namely
\[
\| A(q_2) - A(q_1)\|_{L^2(\T^d)} \leq \frac{3}{4}\|q_2 -q_1\|_{L^2(\T^d)}, \quad q_1,q_2 \in \W_R.
\]
The result is now an immediate consequence of the Banach fixed point theorem, since $\W_R$ is a complete metric space with the distance given by the $L^2$ norm.
\end{proof}

Some comments on this result are in order.
\begin{itemize}[leftmargin=10pt]
\item The exponential rate guarantees a very fast convergence of the iterates to the unknown $q$, and is consistent with the Lipschitz stability of the inverse problem given in \Cref{theo:stab}.
\item We have not presented the details of the corresponding reconstruction algorithm for the Calder\'on problem, which can be easily obtained by using the Liouville transformation $q = \frac{\Delta \sqrt{\sigma}}{\sqrt{\sigma}}$, $v=u/\sqrt{\sigma}$ as in the proof of Corollaries~\ref{cor:uni} and \ref{cor:stab}, in order to formulate Calder\'on problem as an inverse boundary value problem for the Sch\"odinger equation \eqref{eq:schr}.
\end{itemize}

\section{Examples of subspaces $\W$}\label{sec:examples}

As mentioned in \Cref{rem:balancing}, the number $N$ of required measurements to have global uniqueness and stability depends only on the subspace $\W\subseteq L^\infty(\Omega)$ of the potentials. The dependence is explicit via condition \cref{eq:balancing}:
\begin{equation}\label{eq:balancing2}
\|P_N^\perp F P_\W\|_{\H \to \ell^2} \leq \frac 1 4,
\end{equation}
in which $\W$, with an abuse of notation, denotes $i(\W)$, where $i\colon L^\infty(\Omega)\to L^2(\T^d)$ is the extension operator by zero. This condition appears in the literature on signal reconstruction from low frequency Fourier measurements \cite{2014-adcock-hansen-poon,poon2014}: it is strictly related with the balancing property, a fundamental concept in sampling theory and compressed sensing in infinite dimension \cite{AH,2017-adcock-hansen-poon-roman,alberti2017infinite}.

It is worth considering some relevant examples of subspaces $\W$ and to compute the corresponding $N$ as a function  of $\dim\W$. In other words, given $\W$, how many measurements $N$  should we take to have global  uniqueness and stability for the inverse problem?
\subsection{Bandlimited  potentials}
The simplest situation one may consider is with bandlimited potentials $q$ in $L^\infty(\T^d)$ (for simplicity, we set $\Omega=\T^d$). More precisely, the subspace $\W$ is given by 
\[
\W =\{q\in L^\infty(\T^d): \hat q(k) = 0 \;\text{for every $k\in\Z^d$, $\norm{k}_\infty>B$}\},
\]
where $\hat q(k):=\langle q, e_{k} \rangle$ is the Fourier transform and $B\in\N$. In other words, we have
\[
\W = \mathrm{span}\{e_k:k\in\Z^d,\,\norm{k}_\infty\le B\}\subseteq L^\infty(\T^d),
\]
so that $\dim\W=(2B+1)^d$.

It is convenient to choose the ordering $\rho\colon\N\to\Z^d$ in such a way that the frequencies in $\{k\in\Z^d:\norm{k}_\infty\le B\}$ come first, namely
\[
\rho(\{1,\dots,\dim\W\})=\{k\in\Z^d:\norm{k}_\infty\le B\}.
\]
Hence, by \eqref{def:DFT} we immediately have $(Fq)_l=0$ for every $q\in\W$ and $l>\dim\W$. As a result, choosing $N=\dim\W$ gives $\|P_N^\perp F P_\W\|_{\H \to \ell^2}=0$, and so \eqref{eq:balancing2} is automatically satisfied. This is the optimal situation: the number of required measurements equals the dimension of the subspace of the unknowns.

\subsection{Piecewise constant potentials}
A relevant case for the applications is of piecewise constant  potentials \cite{beretta2013} (see \cite{2005-alessandrini-vessella,2006-rondi,alessandrini2017} and references therein for related results).

We consider here a particular situation. Let  $R_1,\dots,R_M\subseteq\Omega$ be $M$ subdomains such that:
\begin{itemize}
\item each subdomain $R_i$ is a $d$-dimensional interval 
with side lengths
\[
a_1^i M^{-\frac{1}{d}},\dots,a_d^i M^{-\frac{1}{d}}
\]
where the scaling $M^{-\frac{1}{d}}$ is put since the size of $\Omega$ is of order $1$;
\item the weights $a_j^i$ are such that $A\le a_j^i$ for some $A\in(0,M^{\frac1d}/\pi]$ and allow for different shapes of the subdomains;
\item and the interiors of these subdomains are disjoint, namely $\mathring{R}_{i_1}\cap \mathring{R}_{i_2}=\emptyset$ for every $i_1\neq i_2$.
\end{itemize}
The subspace $\W$ is given by
\[
\W=\mathrm{span}\{\chi_{R_1},\dots,\chi_{R_M}\},
\]
where $\chi_R$ is the characteristic function of $R$.

The other important ingredient of \eqref{eq:balancing2} is the ordering of the frequencies $\rho\colon\N\to\Z^d$, $l\mapsto k_l$. Here we suppose that \emph{$\rho$ corresponds to the hyperbolic cross in $\Z^d$} \cite[Example~5.12]{2016arXiv161007497J}, namely
\begin{equation}\label{eq:hyperbolic}
l_1\le l_2\quad \implies \quad\prod_{j=1}^d\max(|\rho(l_1)_j|,1) \le  \prod_{j=1}^d\max(|\rho(l_2)_j|,1).
\end{equation}
Recall that the Fourier transform $F\colon L^2(\T^d)\to \ell^2$ is defined by \eqref{def:DFT}, where the frequencies are ordered according to $\rho$.

We now prove that the number of measurements needed to satisfy \eqref{eq:balancing2} (and so to have global uniqueness for the inverse problem considered) is proportional to $M^4$, up to log factors. It is worth observing that this is only a sufficient condition, and may not be necessary. 
Indeed, the use of the Cauchy--Schwarz inequality in \cref{eq:CS} below  yields an additional $M$ factor, which  could perhaps be removed arguing as in  \cite[Lemma 5.1]{poon2014}, at least for a uniform partition of $\T^d$ made of $d$-cubes. The search for the optimal exponent goes beyond the scopes of this work, and is an interesting direction for future research.
\begin{proposition}
Under the above assumptions, we have
\begin{equation}\label{eq:balancing3}
\|P_N^\perp F P_\W\|_{\H \to \ell^2} \leq C  \frac{\log^{d-1}(N)}{\sqrt{N}} M^2,
\end{equation}
for some $C>0$ depending only on $d$, $\rho$ and $A$. In particular, \eqref{eq:balancing2} is satisfied provided that
\[
\frac{N}{\log^{2d-2}(N)}\ge 16C^2M^4.
\]
\end{proposition}
\begin{proof}
Setting $\sinc(x)=\sin(x)/x$, a direct calculation shows that
\[
|F\chi_{R_i}(l)|=\prod_{j=1}^d |\sinc(\pi M^{-\frac{1}{d}} a_j^i \rho(l)_j)|
\le \prod_{j=1}^d \min(\frac{1}{|\pi M^{-\frac{1}{d}} a_j^i \rho(l)_j|},1).
\]
Thus, we readily obtain
\[
|F\chi_{R_i}(l)|
\le \prod_{j=1}^d \frac{1}{\max(|\pi M^{-\frac{1}{d}} A\, \rho(l)_j|,1)}=
\prod_{j=1}^d \frac{(A\pi)^{-1} M^{\frac{1}{d}} }{\max(| \rho(l)_j|,(A\pi)^{-1} M^{\frac{1}{d}} )}.
\]
Hence, in view of \eqref{eq:hyperbolic}, we can apply \cite[Lemma~5.13]{2016arXiv161007497J} and obtain
\begin{equation}\label{eq:chi}
|F\chi_{R_i}(l)|
\le 
(A\pi)^{-d}M \frac{ 1 }{\prod_{j=1}^d\max(| \rho(l)_j|,1 )}\le C(A,d,\rho) M \frac{\log^{d-1}(l+1)}{l} .
\end{equation}

Take now $f\in\W$ with $\norm{f}_{L^2(\T^d)}=1$. Since
\[
\{f_i=\frac{\sqrt{M}}{\sqrt{a_1^i\cdots a_d^i}}\chi_{R_i}:i=1,\dots,M\}
\]
is an orthonormal basis of $\W$, we can write $f=\sum_{i=1}^M c_i f_i$ with $\sum_i c_i^2=1$. Thus we have
\begin{equation}\label{eq:CS}
|Ff(l)|\le \sum_{i=1}^M |c_i| \frac{\sqrt{M}}{\sqrt{a_1^i\cdots a_d^i}} |F \chi_{R_i}(l)|
\le \sqrt{M/A^d}\left(\sum_{i=1}^M |F \chi_{R_i}(l)|^2\right)^\frac12.
\end{equation}
Therefore, \eqref{eq:chi} immediately yields
$
|Ff(l)|\le C(A,d,\rho) M^2 \frac{\log^{d-1}(l+1)}{l}
$, and so
\[
\begin{split}
\norm{P_N^\perp F f}_{\ell^2}^2& = \sum_{l=N+1}^{+\infty}  |Ff(l)|^2\\&\le C(A,d,\rho)M^4 \sum_{l=N+1}^{+\infty} \frac{\log^{2d-2}(l+1)}{l^2}
\\&\le C(A,d,\rho)M^4  \frac{\log^{2d-2}(N)}{N}.
\end{split}
\]
Finally, this bound immediately implies \eqref{eq:balancing3}.
\end{proof}

\subsection{Potentials belonging to low-scale wavelet subspaces}

Given the importance of wavelets in imaging, it is interesting to look at the case when $\W$ is a subspace of dimension $M$ given by wavelets below a certain scale. Under certain assumptions on the mother wavelet and the scaling function, one has $N=O(M)$ if $d=1$ \cite[Lemma 5.1]{poon2014}. This result is expected to hold also in higher dimension, at least in the case of separable wavelets, for which the 1D proof should be easily generalizable (see also \cite{2016arXiv161007497J}). Therefore, if the unknown potential belongs to the space generated by the first $M$ wavelets (ordered according to the scale), $O(M)$ measurements are needed for the reconstruction (up to log factors), and so this estimate is substantially the best possible. It is worth observing that much fewer measurements are needed in this case than in the piecewise constant case.

\bibliography{CS}
\bibliographystyle{plain}

\end{document}